\documentclass[a4paper,11pt,reqno]{amsart}
\usepackage{amsmath}
\usepackage{amsfonts}
\usepackage{amssymb}
\usepackage{graphicx}
\usepackage{amsthm}
\usepackage{color}
\usepackage{url}
\usepackage{bigints}

\newtheorem{theorem}{Theorem}

\newtheorem{definition}[theorem]{Definition}

\newtheorem{lemma}[theorem]{Lemma}

\newtheorem{remark}[theorem]{Remark}

\makeatletter
\newcommand{\addresseshere}{%
  \enddoc@text\let\enddoc@text\relax
}
\makeatother

\newcommand{\rr}{\mathbb{R}}
\renewcommand{\ss}{\mathbb{S}}

\newcommand{\pr}{\partial_{r}}
\newcommand{\pn}{\partial_{\nu}}
\newcommand{\M}{\mathbb{M}}
\newcommand{\Hess}{\mathrm{Hess}}
\newcommand{\ric}{\mathrm{Ric}}
\newcommand{\dive}{\mathrm{div}}

\newcommand{\s}{\sigma}

\providecommand{\abs}[1]{\left|#1\right|}

\usepackage[margin=1in]{geometry}

\begin{document}

\title[A Serrin-type symmetry result on model manifolds]{A Serrin-type symmetry result on model manifolds: an extension of the Weinberger argument}
\author{ Alberto Roncoroni}
\address{Alberto Roncoroni, Dipartimento di Matematica F. Casorati, Universit\`a degli Studi di Pavia, Via Ferrata 5, 27100 Pavia, Italy}
\email{alberto.roncoroni01@universitadipavia.it}

\date{\today}

\begin{abstract}
We consider the classical $\lq\lq$Serrin's symmetry result" for the overdetermined boundary value problem related to the equation $\Delta u=-1$ in a model manifold of non-negative Ricci curvature. Using an extension of the Weinberger classical argument we prove a Euclidean symmetry result under a suitable $\lq\lq$compatibility'' assumption between the solution and the geometry of the model.
\end{abstract}

\maketitle

\section{Preliminaries and statement of the result}

A classical result obtained by Serrin in \cite{Serrin} is the following: 
\begin{theorem}
Let $\Omega$ be a bounded domain in the Euclidean space $\mathbb{R}^{m}$ whose boundary is of class $C^2$. Suppose that $\Omega$ supports a solution $u \in C^{2}(\Omega)\cap C^{1}(\bar \Omega)$ of the overdetermined problem
\begin{equation}\label{serrinpb}
\begin{cases}
 \Delta u=-1 &\mbox{in } \Omega, \\ u=0 &\mbox{on } \partial\Omega, \\ \partial_{\nu}u=\text{constant} &\mbox{on } \partial\Omega,
\end{cases}
\end{equation}
where $\nu$ denotes the exterior unit normal to $\partial \Omega$. Then $\Omega$ is a ball and $u$ has this specific form
\begin{equation}
u(r)=\dfrac{1}{2m}(b^2-r^2),
\end{equation}
where $b$ is the radius of the ball and $r$ denotes distance from its center.
\end{theorem}
This result is known as the ``Serrin's symmetry result'' or the ``Serrin's rigidity result''. The technique used by Serrin to prove this result is a refinement of the famous \emph{reflection principle} due to Alexandrov in \cite{Alexandrov} and is the so-called ``moving planes method'' together with the Maximum Principle and a new version of the Hopf's boundary point Lemma. In particular Alexandrov introduced this method to prove that a closed (i.e. compact without boundary) hypersurface embedded in the Euclidean space $\mathbb{R}^m$ with constant mean curvature must be a sphere. Moreover in \cite{KP}, Kumaresan and Prajapat used the same method of the moving planes to prove the analogous of the ``Serrin's symmetry result'' in the case of bounded domains of the hyperbolic space $\mathbb{H}^{m}$ and of the hemisphere $\mathbb{S}^{m}_{+}$.$\\ $ 
We mention that the technique of Serrin inspired the study of various properties and symmetry results for positive solutions of elliptic partial differential equations in bounded and unbounded domains of the Euclidean space (see the seminal paper by Gidas, Ni and Nirenberg \cite{GNN}). $\\ $
In this article we focus on the more analytic approach by Weinberger \cite{Weinberger} which is based on the Maximum Principle, the integration by parts, the Cauchy-Schwarz inequality and the Bochner formula. We try to extend his proof to the so-called model manifolds with non-negative Ricci curvature.$\\ $
We mention that the approach of Weinberger inspired several works in the context of elliptic
partial differential equations (see e.g. \cite{CGS,CV2,FK,FV,FGK,GL,NT,P,S} and their references).$\\ $ 
In general, as we will see, the importance and the convenience of the model manifolds lies in the fact that their geometry and some natural differential operators (such as the Laplacian) have a particularly simple and explicit description.\medskip

First of all we recall the definition of the $m$-dimensional model manifold 

\begin{definition}
A Riemannian manifold $(\M^{m}_{\s}, g_{\M^{m}_{\s}})$ is called a model manifold if:
\begin{equation*}
\M^{m}_{\s}:=\dfrac{[0,R) \times \ss^{m-1}}{\sim} \quad \textbf{•}{and} \quad g_{\M^{m}_{\s}}:=dr\otimes dr + \s^{2}(r) g_{\ss^{m-1}};
\end{equation*}
where $R\in(0,+\infty]$, $\sim$ is the relation that identifies all the points of $\lbrace 0\rbrace\times\mathbb{S}^{m-1}$ and $\s:[0,R)\rightarrow[0,+\infty)$ is a smooth function such that:
\begin{itemize}
\item $\s(r) >0$, for all $r>0$;
\item $\s^{(2k)}(0)=0$, for all $k= 0,1,2,\dots$;
\item  $\s'(0) =1$.
\end{itemize}
Moreover the unique point corresponding to $r=0$ is called the pole of the model and denoted by $o\in\mathbb{M}^{m}_{\s}$ and $\sigma$ is called the warping function.
\end{definition}
Important examples of model manifolds are the so called space-forms: $\mathbb{R}^m$, $\mathbb{H}^m$ and $\mathbb{S}^m$. Explicitly
\begin{itemize}
\item The Euclidean space $\mathbb{R}^m$ is isometric to the model manifold $\M^{m}_{\s}$ with $\s(r)=r:[0,+\infty)\rightarrow[0,+\infty)$.
\item The iperbolic space $\mathbb{H}^m$ is isometric to the model manifold $\M^{m}_{\s}$ with $\s(r)=\sinh(r):[0,+\infty)\rightarrow[0,+\infty)$.
\item The standard sphere $\mathbb{S}^m\setminus{\lbrace N\rbrace}$ is isometric to the model manifold $\M^{m}_{\s}$ with $\s(r)=\sin(r):[0,\pi)\rightarrow[0,+\infty)$.
\end{itemize}\medskip

We also recall that in $\M^{m}_{\s}$ the Ricci curvature has the following explicit expression. Given $x\in\M^{m}_{\s}$ and $X\in\nabla r(x)^{\perp}$ in $T_{x}\M^{m}_{\s}$ a unit vector we have 
\begin{equation*}
\ric_{\M_{\s}^m}(X,X)=(m-2)\dfrac{1-(\s')^2}{\s^2}-\dfrac{\s''}{\s},
\end{equation*}
and
\begin{equation*}
\ric_{\M_{\s}^m}(\nabla r,\nabla r)=-(m-1)\dfrac{\s''}{\s}.
\end{equation*}
     
With these preliminaries, the main Theorem of this article is the following

\begin{theorem}\label{teofond}
Let $\Omega\subset\M_{\s}^m$ be a smooth domain with $o\in\Omega$. Assume that $\Omega\Subset B_{\tilde{R}}(o)$ where the ray $\tilde{R}>0$ is such that the following conditions on $\s$ are satisfied on the interval $[0,\tilde{R})$:
\begin{itemize}
 \item [$(a)$] $\ric_{\M_{\s}^m}\geq 0 $, i.e. $\s'' \leq 0$ and $(m-2)\left(1-(\s')^{2} \right) - \s\, \s'' \geq 0$;
 \item [$(b)$] $\s' >0$.
\end{itemize}If $\Omega$ supports a solution $u$ of \eqref{serrinpb} and $u$ satisfies the following ``compatibility'' condition
\begin{equation}\label{teofondhp}
 \int_{\Omega}  \frac{(\s'' \s^{m-1})'}{\s^{m-1}} u^{2}\geq 0
\end{equation}
then we have that $\Omega$ is a Euclidean ball of radius $\rho$ centred in the pole $o$ of the model and $u$ has the specific form: 
\begin{equation}\label{uexplicit}
u(r)=\dfrac{1}{2m}(\rho^2-r^2)
\end{equation}
where $r(x)=dist(x,o)$.
\end{theorem}

\begin{remark}
 {\rm
 We analyse the hypothesis of the Theorem.
 \begin{itemize}
  \item Condition $(b)$ appears in other articles on the subject, see for instance \cite{CV1} by Ciraolo and Vezzoni.
  \item The ``compatibility'' condition \eqref{teofondhp} describes a property of the solution in relation to the geometry of the model. It is automatically satisfied by any solution of \eqref{serrinpb} in the case of the Euclidean space and it can not be reduced to a simple condition on the model, like 
 \begin{equation*}
 (\s'' \s^{m-1})'\geq 0.
 \end{equation*}
 Indeed, in this case, the three conditions are compatible only with the flat case: consider $f(r):= \s''(r) \s^{m-1}(r)$. Then
  $f(0)=0$ and if $f'(r) \geq 0$, i.e. $f(r)$ is non-decreasing, so $f(r) \geq 0$ for $r>0$. But $\s''(r)\leq 0$ according to $(a)$, so we have that $\s''(r) = 0$. In this case the result is well known and is presented in Weinberger's article.
 \item Moreover in \cite{Alessandrini-Magnanini} Alessandrini and Magnanini consider a symmetry result for a overdetermined problem and they assume a ``compatibility'' condition as an integral on the boundary of the domain involving the solution and its gradient.
 \end{itemize}
   }
\end{remark}

%

\begin{remark}\label{condbordobislin}
 {\rm
Observe that, by the Strong Maximum Principle, a solution $u$ of \eqref{serrinpb} is positive in $\Omega$. Moreover since $\pn u=constant\neq 0$ on $\Omega$ we obtain that $\abs{\nabla u}\neq 0$ on $\partial\Omega$ and the smooth hypersurface  $\partial\Omega=\lbrace u=0\rbrace$ has exterior normal given by
\begin{equation*}
\nu=-\dfrac{\nabla u}{\abs{\nabla u}}\mid_{\partial\Omega}.
\end{equation*}
This implies that 
\begin{equation*}
\partial_{\nu}u=-\abs{\nabla u} \text{ on } \partial\Omega.
\end{equation*}

%
 }
\end{remark}

\section{Explicit computations towards the proof of Theorem \ref{teofond}}

The Laplace-Beltrami operator $\Delta$ of $\M_{\s}^{m}$ acts on $C^{2}$-functions $u: \M^{m}_{\s}\to \rr$ as follows:
\begin{align}\label{eq1a}
\Delta u &= \pr^{2} u + (m-1)\frac{\s'}{\s} \pr u + \frac{1}{\s^{2}} \bar \Delta u \\
&= \frac{\pr (\s^{m-1}\pr u)}{\s^{m-1}}+ \frac{1}{\s^{2}} \bar \Delta u .\nonumber
\end{align}
where $\bar \Delta$ denotes the Laplacian on the standard sphere $(\ss^{m-1},g_{\ss^{m-1}})$. Using this expression we obtain:
\begin{lemma}\label{lemma1}
The following general formula holds:
\begin{equation}\label{eq2}
\Delta(\s \, \pr u) =  \s\, \pr \Delta u + 2\s'\, \Delta u + (2-m) \s''\, \pr u.
\end{equation}
\end{lemma}
\begin{remark}
{\rm
 In particular, if $\s(r) = r$ and, hence, $\M^{m}_{\s} = \rr^{m}$, we obtain
\begin{equation}\label{eq3}
 \Delta (r \, \pr u) = r \, \pr \Delta u +2 \Delta u,
\end{equation}
which is the traditional formula used by Weinberger to prove Serrin result.
}
\end{remark}
\begin{proof}
We compute 
\begin{align*}
 \s \pr (\Delta u) =& \s \left\lbrace\pr^{3}u+(m-1)\dfrac{\s''\s-(\s')^{2}}{\s^{2}}\pr u+(m-1)\dfrac{\s'}{\s}\pr^2 u-2\dfrac{\s'}{\s^3}\bar \Delta u + \dfrac{1}{\s^2}\pr(\bar \Delta u)\right\rbrace\\
 =&\s\pr^3 u+(m-2)\s''\pr u+\s''\pr u-2(m-1)\dfrac{(\s')^2}{\s}\pr u+(m-1)\dfrac{(\s')^2}{\s}\pr u+\\
 &+(m+1)\s'\pr^2 u-2\s'\pr^2 u -2\dfrac{\s'}{\s^2}\bar \Delta u+\dfrac{1}{\s}\pr(\bar \Delta u)+\dfrac{1}{\s^2}\bar \Delta(\s\pr u)-\dfrac{1}{\s^2}\bar \Delta(\s\pr u)\\
 =& \Delta(\s\pr u)+(m-2)\s''\pr u-2\s'\Delta u+\dfrac{1}{\s}\pr(\bar \Delta u)-\dfrac{1}{\s^2}\bar \Delta(\s\pr u), 
 \end{align*}
 i.e.
 \begin{equation*}
 \Delta(\s\pr u)=\s \pr (\Delta u)+(2-m)\s''\pr u+2\s'\Delta u.
 \end{equation*}
\end{proof}
Now we focus on the solution $u$ of \eqref{serrinpb} (from now on we put the constant in \eqref{serrinpb} equal to $c$) and we show the following 
\begin{lemma}\label{lemma2}
 Let $\Omega$ and $u$ satisfy \eqref{serrinpb}. Then:
\begin{equation}\label{eq5}
(m+2)\int_{\Omega} u\,\s'  = m c^{2}\int_{\Omega} \s' + \frac{(m-2)}{2}  \int_{\Omega} \frac{(\s'' \s^{m-1})'}{\s^{m-1}} u^{2}.
\end{equation}
\end{lemma}
\begin{remark}
 {\rm
  In particular, if $\s(r) = r$ and, hence, $\M^{m}_{\s} = \rr^{m}$, we obtain 
\begin{equation}\label{eq5a}
 (m+2)\int_{\Omega} u = m c^{2}|\Omega|,
\end{equation}
as in the original Weinberger argument; where $|\Omega|$ is the volume of the domain $\Omega$.
 }
\end{remark}
\begin{proof}
First of all we observe that, in this setting, formula \eqref{eq2} becomes
\begin{equation*}
\Delta(\s \, \pr u) =  - 2\s' + (2-m) \s''\, \pr u.
\end{equation*}
So by Green's Theorem
\begin{align*}
\int_{\Omega}- 2\s'\, u + (2-m) \s''\, \pr u\, u+\s\, \pr u =&\int_{\Omega}\Delta(\s \, \pr u) u-\s\, \pr u\, \Delta u \\
 =&\int_{\partial\Omega}\pn(\s \, \pr u) u-\s\, \pr u\, \pn u \\ 
 =& -\int_{\partial\Omega}\s\, (\pn u)^2\pn r\\
 =&-c^2\int_{\partial\Omega}\s\, \pn r\\
 =&-c^2\int_{\Omega}\s\, \Delta r+g_{\M^m_\s}(\nabla r,\nabla \s)\\
 =&-c^2\int_{\Omega}\s\, (m-1)\dfrac{\s'}{\s}+\s'\\
 =&-c^2m\int_{\Omega}\s',
 \end{align*}
 where we have used the fact that $u=0$ on $\partial\Omega$ and that $\pn u=c$ on $\partial\Omega$.$\\ $
Now note that
\begin{align*}
\int_{\Omega} \s \pr u &= \int_{\Omega} g_{\M^m_\s}(\nabla u, \nabla ( \int_{0}^{r}\s(s)ds))\\
&= - \int_{\Omega} u \Delta (\int_{0}^{r}\s(s)ds)\\
&= -m \int_{\Omega} u \s'.
\end{align*}
Using this and the previous computation we have
 \begin{equation}\label{eq6}
  (m+2)\int_{\Omega} u\,\s'  = m c^{2}\int_{\Omega} \s' + (2-m) \int_{\Omega} \s'' \, u \pr u.
\end{equation}
Finally we observe that
\begin{align}\label{eq7}
 \int_{\Omega} \s'' \, u \pr u &= \int_{\Omega}g_{\M^m_\s}(\nabla \s', \nabla (\frac{1}{2}u^{2}))\\\nonumber
 &= -  \frac{1}{2} \int_{\Omega} \Delta \sigma' \, u^{2}\\
 &= - \frac{1}{2} \int_{\Omega} \frac{(\s'' \s^{m-1})'}{\s^{m-1}} u^{2}, \nonumber
 \end{align}
where the second and the third equations are obtained using the condition $u= 0$ on $\partial \Omega$ and the expression \eqref{eq1a}, respectively.$\\ $
 \end{proof}

\section{Proof of Theorem \ref{teofond}}
Now we are ready to prove the main result of this paper.

\begin{proof}[Proof of Theorem \ref{teofond}]
Let $u$ and $\Omega$ as in the statement of Theorem \ref{teofond}; by the Bochner formula and the Cauchy-Schwarz inequality we get
\begin{align}\label{eq9}
 \Delta \left( m  |\nabla u|^{2} + 2 u\right) &= 2m |\Hess(u)|^{2}+2m\ric_{\M^{m}_{\s}}(\nabla u,\nabla u)+2\Delta u \\\nonumber
 &\geq  2(m |\Hess(u)|^{2} +\Delta u)\\\nonumber
 &=2\left( m |\Hess(u)|^{2} - (\Delta u)^{2} \right) \\\nonumber
 &\geq 0 \text{ on }\Omega,\nonumber
 \end{align}
 and the equality holds if and only if 
\begin{equation*}
\Hess(u)=\dfrac{\Delta u}{m}g_{\M^{m}_{\s}}
\end{equation*} 
 and
 \begin{equation*}
\ric_{\M^{m}_{\s}}(\nabla u,\nabla u)=0.
\end{equation*} 
Since, according to Remark \ref{condbordobislin},
\begin{equation}\label{eq10}
  \left(m |\nabla u|^{2} + 2 u\right) = mc^{2} \text{ on } \partial\Omega,
\end{equation}
we conclude from the Strong Maximum Principle that either
\begin{equation}\label{eq11}
\left( m  |\nabla u|^{2} + 2 u\right) < m c^{2} \text{ on }\Omega
\end{equation}
or
\begin{equation}\label{eq12}
  \left( m  |\nabla u|^{2} + 2 u\right) \equiv  m c^{2} \text{ on }\Omega.
\end{equation}
By contradiction assume that condition \eqref{eq11} is satisfied. According to $(b)$ we can multiply both the members of \eqref{eq11} by $\s'$ and integrate in order to obtain
\begin{equation}\label{eq13}
 m \int_{\Omega} |\nabla u|^{2}\s' + 2 \int_{\Omega} u\, \s' < mc^{2} \int_{\Omega} \s'.
\end{equation}
Now we use the identity \eqref{eq5} to deal with the second term i.e.
 \begin{equation}\label{eqfinale}
 2\int_{\Omega} u\,\s'  = m c^{2}\int_{\Omega} \s' + \frac{(m-2)}{2}  \int_{\Omega} \frac{(\s'' \s^{m-1})'}{\s^{m-1}} u^{2}-m\int_{\Omega}u\s'.
 \end{equation}
Note that, by the divergence theorem,
\begin{equation}\label{eq14}
m  \int_{\Omega} \s' \dive(u \nabla u) =  - m \int_{\Omega} \s'' u\, \pr u.
\end{equation}
Moreover,
\begin{equation*}
m  \int_{\Omega} \s' \dive(u \nabla u)=m  \int_{\Omega} \s'|\nabla u|^{2}-m  \int_{\Omega} \s'u.
\end{equation*}
So
\begin{equation}\label{eqfinalebis}
m  \int_{\Omega} \s'|\nabla u|^{2}=m  \int_{\Omega} \s'u- m \int_{\Omega} \s'' u\, \pr u.
\end{equation}
Substituting \eqref{eqfinale} and \eqref{eqfinalebis} in \eqref{eq13} we obtain
\begin{equation*}
 - m \int_{\Omega} \s'' u\, \pr u+m  \int_{\Omega} \s'u+ m c^{2}\int_{\Omega} \s' + \frac{(m-2)}{2}  \int_{\Omega} \frac{(\s'' \s^{m-1})'}{\s^{m-1}} u^{2}-m\int_{\Omega}u\s'<mc^{2} \int_{\Omega} \s'.
 \end{equation*}
Lastly, we use the identity \eqref{eq7} to deduce
 \begin{equation*}
 \frac{m}{2}\int_{\Omega} \frac{(\s'' \s^{m-1})'}{\s^{m-1}} u^{2}+\frac{(m-2)}{2}  \int_{\Omega} \frac{(\s'' \s^{m-1})'}{\s^{m-1}} u^{2}<0,
 \end{equation*}
 i.e.
\begin{equation}\label{eq15}
- (m-1) \int_{\Omega}  \frac{(\s'' \s^{m-1})'}{\s^{m-1}} u^{2}>0;
\end{equation}
and this contradicts the ``compatibility'' condition \eqref{teofondhp}.$\\ $
Therefore \eqref{eq12} holds true and $m  |\nabla u|^{2} + 2 u$ must be constant in $\Omega$. Since its Laplacian then vanishes, we conclude from \eqref{eq9} that equality must hold in Cauchy-Schwarz inequality, i.e. we have proved that $u$ is a solution of (recall that $\Delta u=-1$ in $\Omega$) 
\begin{equation}\label{eq8}
 \Hess(u) = -\frac{1}{m}g_{\M^{m}_{\s}}\text{ in }\Omega.
\end{equation} 
Now, let $\rho:=dist(o,\partial\Omega)$ and take $B_{\rho}(o)\subset\Omega$. Since $\partial\Omega$ is compact, there exists $p\in\partial\Omega$ such that $p\in\partial\Omega\cap\partial B_{\rho}(o)$. In particular, since $u=0$ on $\partial\Omega$, we have that 
\begin{equation*}
u(p)=0.
\end{equation*}
If we prove that $u$ is a radial function in $B_{\rho}(o)$ then 
\begin{equation*}
u=0 \text{ on } \partial B_{\rho}(o).
\end{equation*}
On the other hand, by the Strong Maximum Principle, 
\begin{equation*}
u>0 \text{ in } \Omega.
\end{equation*}
Therefore we can conclude that $\partial B_{\rho}(o)\cap\Omega=\emptyset$ and, hence, $\Omega=B_{\rho}(o)$.$\\ $
So the keypoint is to prove that $u:B_{\rho}(o)\rightarrow\mathbb{R}$, solution of \eqref{eq8}, is a radial function in $B_{\rho}(o)$. 
To this end, take $x\in B_{\rho}(o)$. Since $\M_{\sigma}^{m}$ is geodesically complete there exist a minimizing and normalized geodesic $\gamma\subset B_{\rho}(o)$ from $o$ to $x$. Let $y(t):=u\circ\gamma(t)$ and note that, along $\gamma$, equation \eqref{eq8} implies

\begin{align*}
y''(t)&=\dfrac{d^2}{dt^2}(u\circ\gamma)(t) \\
&=\dfrac{d}{dt}g_{\M^{m}_{\s}}(\nabla u(\gamma(t)),\dot{\gamma}(t)) \\
&=g_{\M^{m}_{\s}}(D_{\dot{\gamma}}\nabla u(\gamma(t)),\dot{\gamma}(t))+g_{\M^{m}_{\s}}(\nabla u(\gamma(t)),D_{\dot{\gamma}}\dot{\gamma}(t))\\
&=g_{\M^{m}_{\s}}((D_{\dot{\gamma}(t)}\nabla u)(\gamma(t)),\dot{\gamma}(t))\\
&=\Hess(u)\mid_{\gamma(t)}(\dot{\gamma}(t),\dot{\gamma}(t))\\
&=-\dfrac{1}{m}.
\end{align*}
The solutions of $y''(t)=-\dfrac{1}{m}$ are given by 
\begin{equation*}
y(t)=-\dfrac{1}{2m}t^2+\alpha t+\beta
\end{equation*}
where $\alpha, \beta\in\mathbb{R}$. Now taking $t=r(x)$ we get
\begin{equation}\label{uradial}
u(x)=u\circ\gamma(r(x))=y(r(x))=-\dfrac{1}{2m}r(x)^2+\alpha r(x)+\beta
\end{equation}
which is radial. To determine the two constant in \eqref{uradial} we recall that $u$ satisfies the following
\begin{equation*}
\begin{cases}
 u(\rho)=0 \\
 u(r) > 0 &\mbox{for } 0<r<\rho
\end{cases}
\end{equation*}
i.e., using the explicit formula of $u$ we obtain
\begin{equation*}
\begin{cases}
-\dfrac{1}{2m}\rho^2+\alpha \rho+\beta=0 \\ 
 \\
 -\dfrac{1}{2m}\left(\dfrac{\rho}{2}\right)^2+\alpha \dfrac{\rho}{2}+\beta > 0 &\mbox{for } r=\dfrac{\rho}{2}
\end{cases}
\end{equation*}
substituting the expression $\beta=\dfrac{1}{2m}\rho^2-\alpha \rho$ in the second equation we get
\begin{equation*}
\alpha<\dfrac{3}{4m}\rho.
\end{equation*}
But, since $u$ must be a $C^2$-function we have that $\alpha=0$; indeed, if we consider the Euclidean case where $r(x)=d(x,0)=\abs{x}$ the gradient of $u$ becomes
\begin{equation}\label{gradEuclideo}
\nabla u(x)=-\dfrac{1}{m}x+\alpha\dfrac{x}{\abs{x}}
\end{equation}
which is not a $C^1$ function in the origin (i.e. the pole of the Euclidean space) unless $\alpha=0$. In a generic model the expression \eqref{gradEuclideo} holds in a system of normal coordinates in the pole. So the same conclusion holds and $\beta=\dfrac{1}{2m}\rho^2$; with this constants the function $u$ becomes
\begin{equation*}
u(r)=-\dfrac{1}{2m}r^2+\dfrac{1}{2m}\rho^2
\end{equation*}
which is exactly the expression \eqref{uexplicit}; observe that, since $u$ is radial, $\pn u=u'(r)$ and the condition $\pn u=constant$ in $\partial\Omega=\partial B_{\rho}(o)$ is automatically satisfied.$\\ $
Moreover we recall that if $f:\mathbb{M}^{m}_{\s}\rightarrow\mathbb{R}$ is a smooth radial function, then its Hessian takes the following expression 
\begin{equation}
\Hess(f)=f''dr\otimes dr+f'\s\s'g_{\mathbb{S}^{m-1}}.
\end{equation}
Using this expression with the function $u$ we get
\begin{equation}
\Hess(u)=-\dfrac{1}{m}dr\otimes dr-\dfrac{1}{m}r\s\s'g_{\mathbb{S}^{m-1}},
\end{equation}
and using this latter in \eqref{eq8} we obtain
\begin{equation*}
-\dfrac{1}{m}dr\otimes dr-\dfrac{1}{m}r\s\s'g_{\mathbb{S}^{m-1}}=-\dfrac{1}{m}(dr\otimes dr+\s^2g_{\mathbb{S}^{m-1}})
\end{equation*}
i.e. 
\begin{equation*}
-\dfrac{1}{m}r\s\s'g_{\mathbb{S}^{m-1}}=-\dfrac{1}{m}\s^2g_{\mathbb{S}^{m-1}}.
\end{equation*}
It follows that $\sigma(r)=r$, so in the ball $B_{\rho}(o)$ not only the solution of \eqref{eq8} is radial but also the metric $g_{\M^{m}_{\s}}$ is the Euclidean metric. This implies that the ball $B_{\rho}(o)$ is a Euclidean ball and the claim follows.

\end{proof}

\begin{remark}[An alternative end of the proof]
{\rm
From the equality sign in the Bochner inequality \eqref{eq9} we get
\begin{equation}
\ric_{\M_{\s}^m}(\nabla u,\nabla u)=0.
\end{equation}
From the explicit expression of $u$ (formula \eqref{uexplicit}) we see that the only critical point is in $r=0$, i.e. in the pole $o$ of the model. So the condition on the Ricci curvature becomes 
\begin{equation}
\ric_{\M_{\s}^m}(\nabla r,\nabla r)=0 \text{ in } B_{\rho}(o)\setminus\lbrace o\rbrace.
\end{equation}
From the explicit expression of $\ric_{\M_{\s}^m}(\nabla r,\nabla r)$ we get $\s''=0$ in $(0,\rho)$ and we conclude that $\s(r)=r$, i.e. $B_{\rho}(o)$ is an Euclidean ball. 
}
\end{remark}

\begin{remark}
{\rm
In \cite{Ros} by Ros we can find a similar spirit where, using the Reilly's formula, he obtained a generalization of Alexandrov theorem for compact hypersurfaces with constant higher order mean curvatures; in this article equation \eqref{eq8} is used to prove a Euclidean symmetry result on a generic compact Riemannian manifold of non-negative Ricci curvature with smooth boundary with mean curvature positive everywhere.
}
\end{remark}

\begin{remark}
{\rm
In this remark we provide an example that shows that if the ``compatibility'' condition \eqref{teofondhp} is not satisfied then we can not have Euclidean symmetry. According to the result of Kumaresan and Prajapat \cite{KP} we know that if we take a domain $\Omega\subset\mathbb{S}^m$ such that $\bar{\Omega}\subset\mathbb{S}^m_+$ and there exist a solution $u$ to the Serrin's symmetry problem \eqref{serrinpb} then $\Omega$ must be a geodesic ball and $u$ must be radially symmetric. We know that the hemisphere is isometric to the model $\mathbb{M}^m_\s$ with $\s(r)=\sin(r)\mid_{[0,\pi/2]}$; so in this example conditions $(a)$ and $(b)$ of Theorem \ref{teofond} are clearly satisfied and the ``compatibility'' condition \eqref{teofondhp} becomes:
\begin{equation*}
 \int_{\Omega}  \frac{(\s'' \s^{m-1})'}{\s^{m-1}} u^{2}=\int_{\Omega}  -m\cos(r) u^{2}(r),  
\end{equation*}
which is negative due to the monotonicity of the integral and to the fact that the function $r\mapsto\cos(r)u^2(r)$ is positive in $\Omega$. $\\ $
In conclusion the ``compatibility'' condition is not satisfied and the symmetry result is not Euclidean since the ball $\Omega$ is a geodesic ball, i.e. the metric in this ball is the metric of the sphere. 
}
\end{remark}




\medskip 

\noindent{\textbf{Acknowledgement.}} The author wishes to thank Stefano Pigola for his precious help and useful discussions. The author has been supported by the ``Gruppo Nazionale per l'Analisi Matematica, la Probabilit\`a e le loro Applicazioni'' (GNAMPA) of the ``Istituto Nazionale di Alta Matematica'' (INdAM).

\addresseshere


\medskip

\begin{thebibliography}{20} 

\bibitem{Alessandrini-Magnanini} G. Alessandrini, R. Magnanini, Symmetry and non-symmetry for the ovedetermined Stekoff problem II. \emph{Nonlinear Problems in Applied Mathematics}, L. P. Cook et al. eds., SIAM, Philadelphia (1995).

\bibitem{Alexandrov} A. D. Alexandrov, Uniqueness theorem for surfaces in large I. \emph{Vestnik Leningrad Univ. Math.} \textbf{11} (1956), 5-17.


\bibitem{CGS} L. Caffarelli, N. Garofalo, and F. Seg\`ala, A gradient bound for entire solutions of quasi-linear equations and its consequences, \emph{Comm. Pure Appl. Math.}, \textbf{47} (1994), 1457-1473.

\bibitem{CV1} G. Ciraolo, L. Vezzoni, A rigidity problem on the round sphere. \emph{Commun. Cont. Math.}  \textbf{19} (2016), 1750001. 

\bibitem{CV2} G. Ciraolo, L. Vezzoni, On Serrin's overdetermined problem in space forms. Preprint (arXiv:1702.05277)



\bibitem{FK} A. Farina, B. Kawohl, Remarks on an overdetermined boundary value problem. \emph{Calc. Var. Partial Differential Equations}  \textbf{31} (2008), 351-357.

\bibitem{FV} A. Farina, E. Valdinoci, A pointwise gradient estimate in possibly unbounded domains with nonnegative mean curvature. \emph{Adv. Math.}, 225 (2010), no. 5, 2808-2827.

\bibitem{FGK} I. Fragal\`a, F. Gazzola and B. Kawohl, Overdetermined problems with possibly degenerate ellipticity, a geometric approach. \emph{Math. Z.} \textbf{254} (2006), 117-132.

\bibitem{GL} N. Garofalo, J.L. Lewis, A symmetry result related to some overdetermined boundary value problems, \emph{Amer. J. Math.} \textbf{111} (1989), 9-33.

\bibitem{GNN} B. Gidas, W. M. Ni and L. Nirenberg, Symmetry and related properties via the maximum principle, \emph{Comm. Math. Phys.} \textbf{68} (1979), no. 3, 209-243.

\bibitem{KP} S. Kumaresan, J. Prajapat, Serrin's result for hyperbolic space and sphere. \emph{Duke Math. J.} \textbf{91} (1998), 17-28.

\bibitem{NT} C. Nitsch, C. Trombetti, The classical overdetermined Serrin problem. Complex Variables and Elliptic Equations (2017).

\bibitem{P} L.E. Payne, Some remarks on maximum principles, \emph{J. Anal. Math.}, \textbf{30} (1976) 421-433.

\bibitem{Petersen} P. Petersen, \textit{Riemannian Geometry}.  †Grad. Texts in Math. 71. Springer-Verlag, New York, 1998.

\bibitem{Ros} A. Ros, Compact Hypersurfaces with Constant Higher Order Mean Curvature. \emph{Rev. Mat. Iberoam.} \textbf{3} (1987), 447-453.

\bibitem{Serrin} J. Serrin, A symmetry problem in potential theory. \emph{Arch. Rational Mech. Anal.} \textbf{43} (1971), 304-318.

\bibitem{S} R. P. Sperb, Maximum Principles and Their Applications, Mathematics in Science and Engineering, vol. 157, Academic Press Inc. [Harcourt Brace Jovanovich Publishers], New York, 1981.

\bibitem{Weinberger} H. Weinberger, Remark on the preceeding paper of Serrin. \emph{Arch. Rational Mech. Anal.} \textbf{43} (1971), 319-320.

\end{thebibliography}
\end{document}